\documentclass[11pt]{amsart}
\usepackage[T1]{fontenc}
\usepackage{lmodern}
\usepackage{amssymb}
\usepackage[margin=1.05in]{geometry}
\usepackage{microtype}
\usepackage{enumitem}
\usepackage{needspace}
\usepackage[colorlinks=true,linkcolor=blue,citecolor=blue,urlcolor=blue]{hyperref}
\hypersetup{pdftitle={Traces of powers of canonical ideals},pdfauthor={Le Truong Hoang; Naoyuki Matsuoka}}
\newtheorem{theorem}{Theorem}[section]
\newtheorem{proposition}[theorem]{Proposition}
\newtheorem{lemma}[theorem]{Lemma}
\newtheorem{corollary}[theorem]{Corollary}
\theoremstyle{definition}
\newtheorem{notation}[theorem]{Notation}
\newtheorem{example}[theorem]{Example}
\newtheorem{question}[theorem]{Question}
\newtheorem{conjecture}[theorem]{Conjecture}
\newcommand{\m}{\mathfrak m}
\newcommand{\fkc}{\mathfrak c}
\newcommand{\tr}{\operatorname{tr}_R}
\newcommand{\Hom}{\operatorname{Hom}}
\newcommand{\Soc}{\operatorname{Soc}}
\newcommand{\PF}{\operatorname{PF}}
\newcommand{\type}{\operatorname{r}}
\newcommand{\ord}{\operatorname{ord}}
\newcommand{\Frac}{\operatorname{Q}}
\title{Traces of powers of canonical ideals}
\author{Le Truong Hoang}
\address[Le Truong Hoang]{University of Engineering and Technology, Vietnam National University, Hanoi, Vietnam}
\email{hltruong@vnu.edu.vn}
\author{Naoyuki Matsuoka}
\address[Naoyuki Matsuoka]{Department of Mathematics, School of Science and Technology, Meiji University, 1-1-1 Higashi-mita, Tama-ku, Kawasaki 214-8571, Japan}
\email{naomatsu@meiji.ac.jp}
\date{}
\thanks{The second author was partially supported by JSPS KAKENHI Grant Number JP25K06941.}

\thanks{Work (OpenAI) was used in an iterative drafting process based on the authors' research notes and detailed instructions, and to assist with subsequent revisions of the manuscript. It was also used to assist with searches for related literature, citation checks, and pre-submission review of the manuscript. The authors reviewed and substantially revised the resulting text, verified the mathematical arguments and references, and take responsibility for the final manuscript.}

\begin{document}
\begin{abstract}
Let $(R,\m)$ be a one-dimensional Cohen--Macaulay local ring
with a canonical fractional ideal $R\subseteq K\subseteq\overline R$, and let
$v$ denote its embedding dimension. Put
$\sigma(R)=\sup\{j\ge0\mid\tr(K^j)\supseteq\m\}$.
By definition, $R$ is nearly Gorenstein if and only if $\sigma(R)>0$.
We prove that $R$ is almost Gorenstein if and only if $\sigma(R)\ge v$,
equivalently, $\sigma(R)=\infty$. Thus every finite value of $\sigma(R)$
is at most $v-1$.
We then determine the structure of rings for which $\sigma(R)$ attains
its largest finite value $v-1$, just below the threshold for the almost
Gorenstein property. This occurs if and only if $K^{v-1}\cong\m$ and
$K^v\not\cong\m$.
When these conditions hold, $R$ has type two, with multiplicity $e$ greater than $v$ and $e \equiv 3 \pmod{v-1}$.
As an application, we characterize the numerical semigroup rings $R$ 
with $\sigma(R)=v-1$ as those generated by an arithmetic progression,
of type two and multiplicity greater than $v$.
\end{abstract}
\maketitle

\section{Introduction}

A Cohen--Macaulay local ring with a canonical module is Gorenstein if
and only if its canonical module is isomorphic to the ring.
The almost and nearly Gorenstein properties weaken this condition in
different ways.
Barucci and Fr\"oberg introduced almost Gorenstein rings among
analytically unramified one-dimensional local rings \cite{BF97}.
Goto, Matsuoka, and Phuong extended the definition beyond the
analytically unramified case \cite{GMP13}.
Herzog, Hibi, and Stamate introduced nearly Gorenstein rings through
the trace of the canonical module \cite{HHS19}. They showed that every
one-dimensional almost Gorenstein ring is nearly Gorenstein
\cite[Proposition~6.1]{HHS19}, whereas the converse does not hold in
general.

Throughout this paper, we maintain the following notation.

\begin{notation}\label{not:base}
Let $(R,\m)$ be a one-dimensional
Cohen--Macaulay local ring with the residue field $k=R/\m$. We always assume that there exists a
fractional ideal $K$, isomorphic to the canonical module of $R$, such that
$R\subseteq K\subseteq\overline R$,
where $\overline R$ denotes the integral closure of $R$ in its total
quotient ring $\Frac(R)$. Put
\[
 v=\mu_R(\m),\qquad e=e_0(R),\qquad S=R[K].
\]
Here $\mu_R(-)$ denotes the minimal number of generators and $e_0(R)$
denotes the multiplicity of $R$ with respect to $\m$.
We write $\type(R)=\mu_R(K)$ for the Cohen--Macaulay type of $R$.
For an $R$-module $M$ of finite length, we write $\ell_R(M)$ for its length.
\end{notation}

Under Notation~\ref{not:base}, $R$ is nearly Gorenstein if and only if
$\m\subseteq\tr(K)$. It is almost Gorenstein if and only if
$\m K\subseteq R$ \cite[Theorem~3.11]{GMP13}.

The algebra $S$ is a finite $R$-module, because $K$ is finitely generated and its elements are integral over $R$. 
For the powers of $K$, we put
\[
 K^0=R,\qquad K^{n+1}=K^nK=\sum_{x\in K}xK^n\subseteq\Frac(R)
 \quad(n\ge0).
\]

Since $R\subseteq K$, the powers of $K$ form an ascending chain that
stabilizes at $S$:
\[
 R\subseteq K\subseteq K^2\subseteq\cdots\subseteq S.
\]
The bound $K^{e-1}=S$ follows from Sally and Vasconcelos
\cite[Lemma~2.4 and Theorem~2.5]{SV74}.
Kumashiro proved that a one-dimensional nearly Gorenstein ring is
almost Gorenstein if and only if its canonical reduction number is at
most two \cite[Theorem~3.12]{Kum22}. In our setting, this number is the
least integer $n\ge0$ for which $K^n=S$.
Isobe and Kumashiro \cite[Theorem~1.1]{IK25} gave another
characterization: $R$ is almost Gorenstein if and only if
$\tr(\m K)\supseteq\m$.

We ask how long $\m\subseteq\tr(K^n)$ persists as $n$ increases. The
traces form a descending sequence:
\[
 R=\tr(K^0)\supseteq\tr(K)\supseteq\tr(K^2)\supseteq\tr(K^3)\supseteq\cdots.
\]
We measure how long these traces contain $\m$ by
\[
 \sigma(R)=\sup\{n\ge0\mid\tr(K^n)\supseteq\m\}.
\]
Since the traces form a descending sequence, $R$ is nearly Gorenstein
if and only if $\sigma(R)>0$.
This index is independent of the choice of $K$. Any other canonical
fractional ideal is $\alpha K$ for a unit $\alpha$ of $\Frac(R)$, and multiplication
by $\alpha^j$ gives an isomorphism between their $j$th powers for every
integer $j\ge1$.

If $R$ is Gorenstein, then $K= R$ in our setting, so $\tr(K^n)=R$ for every
$n\ge0$. In particular, $\sigma(R)=\infty$.
More generally, $\sigma(R)=\infty$ whenever $R$ is almost Gorenstein.
The following proposition shows that the converse holds as well.

\Needspace{8\baselineskip}
\begin{proposition}\label{prop:ag}
Under Notation~\ref{not:base}, the following conditions are equivalent.
\begin{enumerate}[label=\textup{(\arabic*)}]
\item $R$ is almost Gorenstein.
\item $\tr(K^j)\supseteq\m$ for every integer $j\ge1$, i.e., $\sigma(R)=\infty$.
\item $\tr(K^v)\supseteq\m$, i.e., $\sigma(R)\ge v$.
\end{enumerate}
\end{proposition}

If $R$ is not almost Gorenstein, Proposition~\ref{prop:ag} gives
$\sigma(R)\le v-1$. 
Our main theorem determines the structure of the rings for which
$\sigma(R)=v-1$.

If $\sigma(R)=v-1$, then $R$ is not Gorenstein and
$\tr(K^j)\subseteq\m$ for every $j\ge1$ by Lemma~\ref{lem:trace}(1).
Therefore $\sigma(R)=v-1$ is equivalent to the trace conditions in (1) below.

\begin{theorem}\label{thm:main}
Under Notation~\ref{not:base}, the following conditions are equivalent.
\begin{enumerate}[label=\textup{(\arabic*)}]
\item $\tr(K^{v-1})=\m$ and $\tr(K^v)\subsetneq\m$, i.e., $\sigma(R)=v-1$.
\item $K^{v-1}\cong\m$ and $K^v\not\cong\m$ as $R$-modules.
\item There exists a regular element $a\in\m$ such that
      $\m=aK^{v-1}$ and $\m K\ne\m$.
\item $e>v$, and there exist a regular element $a\in\m\setminus\m^2$
and $u\in K$ such that $(a)$ is a reduction of $\m$ and
\[
 K=R+Ru,\qquad
 \m=a(R+Ru+\cdots+Ru^{v-1}),\qquad S=R[u].
\]
\end{enumerate}
When these conditions hold, the following assertions hold true.
\begin{enumerate}[label=\textup{(\alph*)}]
\item $\type(R)=2$, $e>v \ge 3$, and $e\equiv3\pmod{v-1}$.
\item $K\cap\m S=\m K$. Moreover, $K^{e-1}=S$ and $K^n\ne S$ for
      every integer $0\le n <e-1$.
\end{enumerate}
\end{theorem}

For numerical semigroup rings, which are typical examples of one-dimensional Cohen--Macaulay local rings, the expressions in Theorem~\ref{thm:main}(4) have
a concrete meaning. Corollary~\ref{cor:semigroup} shows that the
minimal generators of the semigroup form an arithmetic progression.
Conversely, the ring of a numerical semigroup generated by an arithmetic
progression, of type two and multiplicity greater than $v$, satisfies
the equivalent conditions of Theorem~\ref{thm:main}. The theorem therefore gives a
ring-theoretic description of this class of numerical semigroup rings,
valid without assuming that $R$ is a semigroup ring.
The type of a numerical semigroup generated by a generalized arithmetic
sequence is already determined by the pseudo-Frobenius numbers computed
by Matthews \cite[Lemma~2.7]{Matthews04}. Once arithmetic progression
and type two are known, the congruence $e\equiv3\pmod{v-1}$ follows
from that formula. In Theorem~\ref{thm:main}, the multiplicity assertion
is obtained for general local rings without using a semigroup description.

The determination of the type and the generators from $\sigma(R)=v-1$
makes it natural to consider what remains true when $\sigma(R)\le v-2$. A comparable
description in this range appears more difficult, even for numerical
semigroup rings. We give an example with $\sigma(R)=v-2$ at the end of
the paper.

In Section~2 we prove Proposition~\ref{prop:ag} and prepare for the
proof of Theorem~\ref{thm:main}, which is given in Section~3.
In Section 4 we apply Theorem~\ref{thm:main} to numerical semigroup rings and 
determine when the traces of powers of \(K\) equal \(\mathfrak m\) 
for semigroups generated by arithmetic progressions. We conclude 
with examples, a conjecture concerning a smaller value of \(\sigma(R)\),
and a question concerning ring-theoretic characterizations of arithmetic
progression generation.

\section{Traces and the embedding dimension}

Under Notation~\ref{not:base}, we first describe homomorphisms between
fractional ideals to compute the traces of the powers of $K$.
Let $I,J$ be fractional ideals
of $R$ containing regular elements. Since
$I\otimes_R\Frac(R)=J\otimes_R\Frac(R)=\Frac(R)$, every
$R$-homomorphism $I\to J$ is multiplication by an element of $\Frac(R)$.
Therefore
\[
 \Hom_R(I,J)\cong J:I
 =\{x\in\Frac(R)\mid xI\subseteq J\}.
\]
We use this colon notation throughout the proofs.

For an $R$-module $M$, its trace is defined by
$\tr(M)=\sum_{\varphi\in\Hom_R(M,R)}\varphi(M)$.
In this paper, we take traces only of fractional ideals containing regular elements.
Taking $J=R$ in the description above therefore gives
\[
 \tr(I)=(R:I)I.
\]
We will also use the trace of $\m$. If $R$ is not a discrete valuation ring, then
$R:\m=\m:\m$ by \cite[Lemma~3.15(3)]{GMP13}, and hence
$\tr(\m)=\m$.

The following properties of the traces of $K^j$ will be used to prove
Proposition~\ref{prop:ag}.

\begin{lemma}\label{lem:trace}
The following assertions hold true.
\begin{enumerate}[label=\textup{(\arabic*)}]
\item $\tr(K^{j+1})\subseteq\tr(K^j)$ for every integer $j\ge0$.
      If $R$ is not Gorenstein, then $\tr(K^j)\subseteq\m$ for every $j\ge1$.
\item If $\m K=\m$, then $\tr(K^j)\supseteq\m$ for every integer $j\ge1$.
\end{enumerate}
\end{lemma}
\begin{proof}
(1) The first inclusion follows from \cite[Proposition~2.8(i)]{Lin17}.
Let $j \ge 1$ be an integer and suppose $R$ is not Gorenstein. Since
$R\subseteq K\subseteq K^j$,
$R:K^j$ is a proper ideal of $R$ (otherwise
$K\subseteq K^j\subseteq R$, so $K=R$), hence is contained in $\m$.
As $S$ is integral over $R$, we have $\m S\cap R=\m$. Therefore
$\tr(K^j)=(R:K^j)K^j\subseteq\m S\cap R=\m$.

(2) Let $j \ge 1$ be an integer. Since $\m K^j=\m$ and $\m\subseteq R:K^j$, we have
$\m=\m K^j\subseteq(R:K^j)K^j=\tr(K^j)$.
\end{proof}

The next lemma bounds the exponent of a power whose trace equals $\m$
and gives a special kind of element $a$ when the bound is attained.

\begin{lemma}\label{lem:growth}
Suppose $R$ is not almost Gorenstein.
Then the following assertions hold true.
\begin{enumerate}[label=\textup{(\arabic*)}]
\item For every integer $q\ge1$, if $\tr(K^q)=\m$, then $q\le v-1$.
\item If $\tr(K^{v-1}) = \m$, then there is a regular element $a\in\m\setminus\m^2$ such that
\[
 \m=aK^{v-1},\qquad \m S=aS,\qquad\text{and}\qquad \m\cap a^2S=\m^2.
\]
Furthermore, for this choice of $a$, we have
\[
 \dim_k ((K^i+aS)/aS)=i+1\quad(0\le i\le v-1)\qquad \text{and} \qquad 
 \dim_k(S/aS)>v.
\]
\end{enumerate}
\end{lemma}
\begin{proof}
By \cite[Theorem~3.11]{GMP13} and $R\subseteq K\subseteq S$, the
assumption that $R$ is not almost Gorenstein gives $\m S\ne\m$.

To prove (1), fix an integer $q\ge1$ such that $\tr(K^q)=\m$.
Put $I=R:K^q$ and $J=IS=\m S$; the last equality
follows from $IK^q=\tr(K^q)=\m$ and $K^qS=S$.
Since $J$ and $\m J$ are $S$-modules, $S$ acts on $J/\m J$.
For each integer $0\le i\le q$, let
\[
 F_i=(IK^i+\m J)/\m J
\]
be a subspace of $J/\m J$. Then, for every $0\le i<q$, $KF_i=(IK^{i+1}+\m J)/\m J=F_{i+1}$. 
Moreover, $SF_0=J/\m J$, because $IS=J$. In particular, $F_0\ne0$; otherwise
$J=\m J$, contrary to Nakayama's lemma and $\m\subseteq J$.

Suppose $F_i=F_{i+1}$ for some integer $0\le i<q$. Then
$KF_i=F_i$, so $SF_i=F_i$ because $S=R[K]$. Since $F_0\subseteq F_i$,
\[
 J/\m J=SF_0\subseteq SF_i=F_i\subseteq J/\m J.
\]
Thus $F_i=J/\m J$, and hence $F_q=J/\m J$. The latter equality is
impossible: $F_q=(\m+\m J)/\m J$ would give $J=\m+\m J$, so
Nakayama's lemma applied to $J/\m$ would give $J=\m$, contrary to
$J=\m S\ne\m$.
Consequently,
\[
 0\ne F_0\subsetneq F_1\subsetneq\cdots\subsetneq F_q
 \subsetneq J/\m J.
\]
On the other hand, $F_q\cong\m/(\m\cap\m J)$ and
$\m^2\subseteq\m\cap\m J$. Hence
\[
 q+1\le\dim_kF_q\le\mu_R(\m)=v,
\]
proving (1).

To prove (2), suppose $\tr(K^{v-1})=\m$ and apply the preceding
construction with $q=v-1$.
The dimension count in (1) gives
\[
 \dim_kF_i=i+1\quad(0\le i\le q),\qquad
 \m\cap\m J=\m^2.
\]
Choose $a\in I$ whose image spans $F_0$. Then
$I\subseteq aR+\m J$. Since $\m J$ is an $S$-module,
multiplication by $K^q$ gives $\m\subseteq aK^q+\m J$.
Using $aK^q\subseteq\m$ and $\m\cap\m J=\m^2$, we obtain
\[
 \m=aK^q+(\m\cap\m J)=aK^q+\m^2.
\]
Nakayama's lemma gives $\m=aK^q$.
The ideal $\m$ contains a regular element, so the equality $\m=aK^q$
shows that $a$ is regular. Also $a\in\m$ and $a\notin\m J$;
therefore $a\notin\m^2$. We now have $J=aS$ and $\m J=a^2S$, which
prove the three asserted ideal equalities.

Since $a\in I\subseteq aR+\m J$ and $\m J=a^2S$, we have, for
$0\le i\le q$,
\[
 F_i= (IK^i+\m J)/\m J
    =(aK^i+a^2S)/a^2S
    \cong(K^i+aS)/aS.
\]
Also $J/\m J=aS/a^2S\cong S/aS$. Hence
$\dim_k((K^i+aS)/aS)=i+1$ for $0\le i\le q$, and
$\dim_k(S/aS)=\dim_k(J/\m J)>v$.
\end{proof}

We apply Lemmas~\ref{lem:trace} and~\ref{lem:growth} to prove
Proposition~\ref{prop:ag}.

\begin{proof}[Proof of Proposition~\ref{prop:ag}]
If $R$ is Gorenstein, then $K=R$ and all three conditions hold.
We may therefore assume that $R$ is not Gorenstein. By
Lemma~\ref{lem:trace}(1), for every integer $j\ge1$,
\[
 \tr(K^j)\supseteq\m\quad\Longleftrightarrow\quad\tr(K^j)=\m.
\]
By \cite[Theorem~3.11]{GMP13}, $R$ is almost Gorenstein if and only if
$\m K\subseteq R$, which is equivalent to $\m K=\m$.
Thus (1) implies (2) by Lemma~\ref{lem:trace}, and (2) implies (3).
Suppose (3) holds and $R$ is not almost Gorenstein. Then
$\tr(K^v)=\m$ by Lemma~\ref{lem:trace}(1), but
Lemma~\ref{lem:growth}(1) with $q=v$ gives $v\le v-1$, a contradiction.
Hence (3) implies (1).
\end{proof}

\section{Proof of the main theorem}

We keep the standard notation as in Notation \ref{not:base}.
In this section, we prove Theorem~\ref{thm:main}.
We first prove that (1)--(3) are equivalent and that (4) implies (3).
The main step is to derive (4) from (1)--(3); the same argument also
proves (a) and (b).

If $R$ is Gorenstein, then $K=R$, so none of conditions (1)--(3)
in Theorem~\ref{thm:main} holds.
Condition (4) also fails. Indeed, $K=R$ would give $\m=aR$, so $R$ would
be a discrete valuation ring with $e=v=1$.
We may therefore assume that $R$ is not Gorenstein. In particular,
$v\ge3$, since a one-dimensional Cohen--Macaulay local ring of embedding
dimension at most two is Gorenstein.

(1) $\Rightarrow$ (3): By Proposition~\ref{prop:ag}, $R$ is not almost
Gorenstein. Hence $\m K\ne\m$ and $\m S\ne\m$, so
Lemma~\ref{lem:growth}(2) gives $\m=aK^{v-1}$ for a regular element
$a\in\m\setminus\m^2$.

(3) $\Rightarrow$ (2): Choose $a$ as in (3). Multiplication by $a$ gives $K^{v-1}\cong\m$.
If $K^v\cong\m$, then $\tr(K^v)=\tr(\m)=\m$, and
Proposition~\ref{prop:ag} gives $\m K=\m$, a contradiction.

(2) $\Rightarrow$ (1): Trace is invariant under isomorphism, so
$\tr(K^{v-1})=\tr(\m)=\m$. If $\tr(K^v)=\m$, then
Proposition~\ref{prop:ag} gives $\m K=\m$. Choose an isomorphism
$K^{v-1}\longrightarrow\m$. It is multiplication by a unit $\alpha$ of
$\Frac(R)$, by the description of homomorphisms in Section~2;
hence $\alpha K^v=\m K=\m$, contrary to $K^v\not\cong\m$.
Since $\tr(K^v)\subseteq\m$ by Lemma~\ref{lem:trace}(1),
we obtain $\tr(K^v)\subsetneq\m$.

(4) $\Rightarrow$ (3): Choose $a$ as in (4). We have $\m=aK^{v-1}$. Suppose $\m K=\m$.
Then $\m^2=aK^{v-1}\m=a\m$. Since $a\notin\m^2$, its image is
part of a minimal generating set of $\m$. As $\m^2=a\m\subseteq aR$,
we have
\[
 \ell_R(\m/aR)
 =\dim_k\bigl(\m/(aR+\m^2)\bigr)=v-1.
\]
Consequently,
\[
 e=\ell_R(R/aR)=1+\ell_R(\m/aR)=1+(v-1)=v,
\]
where the first equality holds because $(a)$ is a reduction of $\m$.
This contradicts $e>v$, so $\m K\ne\m$.

For the rest of this section, assume the equivalent conditions (1)--(3).
Put $q=v-1$, and choose $a$ as in Lemma~\ref{lem:growth}(2).
Since $K^n=S$ for $n\gg0$,
\[
 \m^{n+1}=a^{n+1}K^{(n+1)q}=a^{n+1}S=a\m^n
 \qquad(n\gg0).
\]
Thus $(a)$ is a reduction of $\m$ and $e=\ell_R(R/aR)$.
Since $S/R$ has finite length, multiplication by $a$ on $S/R$ has
kernel and cokernel of equal length. Hence
\[
 e=\ell_R(R/aR)=\ell_R(S/aS)=\dim_k(S/aS)>v,
\]
where the last equality uses $\m S=aS$, and the inequality follows
from Lemma~\ref{lem:growth}(2).

For $x\in S$, write $\overline{x}=x+aS\in S/aS$.
We first find an element $u\in K$ whose powers describe $\m$ and $S$.

\begin{lemma}\label{lem:powers}
There exists $u\in K$ such that, with $L=R+Ru$, the following assertions hold.
\begin{enumerate}[label=\textup{(\arabic*)}]
\item
$K^q=L^q$, $\m=aL^q$, $S=R[u]$, and $L^{e-1}=S$.
Moreover, $1,\overline{u},\ldots,\overline{u}^{e-1}$ form a
$k$-basis of $S/aS$.
\item For every integer $n\ge0$,
$\m^n=\sum_{i=0}^{nq}R\,(a^nu^i)=R\cap a^nS$ and $aR\cap\m^{n+1}=a\m^n$.
\end{enumerate}
\end{lemma}
\begin{proof}
Lemma~\ref{lem:growth}(2) gives $\dim_k((K+aS)/aS)=2$.
Choose $u\in K$ such that
\[
 (K+aS)/aS=k+k\overline{u}.
\]
Since $S=R[K]$, we have $S/aS=k[\overline{u}]$, whence $1,\overline{u},\ldots,\overline{u}^{e-1}$ form a $k$-basis of $S/aS$ because $\dim_k S/aS = e$.
Put $L=R+Ru$. 

(1) Since $(K+aS)/aS=k+k\overline{u}$, we have
$K^q+aS=L^q+aS$.
Using $\m=aK^q$ and $\m\cap a^2S=\m^2$, we have
\[
 a(K^q\cap aS)=\m\cap a^2S=\m^2=a\m K^q.
\]
Hence, since $a$ is regular, $K^q\cap aS=\m K^q$. Thus
$K^q=K^q \cap (L^q + aS) = L^q+(K^q\cap aS)=L^q+\m K^q$, and Nakayama's lemma gives
$K^q=L^q$.
Thus $\m=aL^q$. Also $K\subseteq K^q=L^q\subseteq R[u]$, while
$u\in K$, so $S=R[u]$. The basis of $S/aS$ gives
$S=L^{e-1}+aS=L^{e-1}+\m S$. Nakayama's lemma yields $L^{e-1}=S$.

(2) We first show that
\[
 L^d\cap aS=\m L^d
\]
for every integer $d \ge 0$.
For $0\le d\le e-1$, if $\sum_{i=0}^d r_i u^i\in aS$ with
$r_i\in R$, then the independence of $1,\overline{u},\ldots,\overline{u}^d$ gives
$r_i\in\m$ for $0\le i\le d$. This proves
$L^d\cap aS\subseteq\m L^d$; the reverse inclusion follows from
$\m S=aS$. For $d\ge e-1$, we have
$L^d\cap aS=aS=\m S=\m L^d$.

Next, we prove
$$R\cap a^nS=a^nL^{nq}$$ for every $n \ge 0$ by induction on $n$.
When $n=0$, it is obvious. 
If it holds for $n$, then
\[
 \begin{aligned}
 R\cap a^{n+1}S
 &=(R\cap a^nS)\cap a^{n+1}S\\
 &=a^n(L^{nq}\cap aS)
 =a^n\m L^{nq}
 =a^{n+1}L^{(n+1)q}
 \end{aligned}
\]
since $\m = aL^q$. 
Hence, for every $n \ge 0$, we have the desired equality, in particular, $\m^n = (aL^q)^n = a^n L^{nq} = R \cap a^nS$. 
Finally, since $a$ is regular,
\[
 aR\cap\m^{n+1}
 =aR\cap a^{n+1}S
 =a(R\cap a^nS)=a\m^n
\]
for every $n \ge 0$.
\end{proof}

In what follows, fix $u \in K$ as in Lemma~\ref{lem:powers}.
Put $L=R+Ru$ and $\fkc=R:S$. 
We aim to prove $K\cap aS=\m K$.
By \cite[Lemma~3.5(2)]{GMP13}, $\fkc=K:S$.
We use $K:K=R$ from \cite[p.~19, Bemerkung~2.5(a)]{HK71} and the
canonical biduality formula $K:(K:M)=M$ for finite fractional ideals
$M$ containing a regular element; see \cite[p.~19, Definition~2.4(b)]{HK71}.
In particular, $K:\fkc=K:(K:S)=S$. Therefore
\[
 \begin{aligned}
 K:(R+a^{-1}\fkc)&=K\cap a(K:\fkc)=K\cap aS,\\
 K:(R:\m)&=K:(K:\m K)=\m K.
 \end{aligned}
\]
Thus the desired equality $K\cap aS=\m K$ follows from $R:\m=R+a^{-1}\fkc$.
Since $a$ is regular, $\Soc(R/aR)=a(R:\m)/aR$. Hence
$R:\m=R+a^{-1}\fkc$ is equivalent to
\[
 \Soc(R/aR)=(\fkc+aR)/aR.
\]
To prove this socle equality, put $s=\lceil(e-1)/q\rceil$. Then $L^{sq}=S$ by
Lemma~\ref{lem:powers}(1).

\begin{lemma}\label{lem:socle}
The following assertions hold true.
\begin{enumerate}[label=\textup{(\arabic*)}]
\item $\Soc(R/aR)=(\m^s+aR)/aR =(\fkc+aR)/aR$.
\item $\dim_k\Soc(R/aR)=e-1-(s-1)q$.
\end{enumerate}
\end{lemma}
\begin{proof}
(1) Lemma~\ref{lem:powers}(1) and $sq\ge e-1$ give
$\m^s=a^sS$. Thus $\m^sS=\m^s\subseteq R$, so
$\m^s\subseteq\fkc$. Since $\fkc$ is an $S$-ideal and $\m S=aS$ by Lemma~\ref{lem:growth}(2),
\[
 \fkc\m=\fkc\m S=a\fkc\subseteq aR.
\]
In particular, $\m^{s+1}\subseteq aR$, and
\[
 (\m^s+aR)/aR\subseteq
 (\fkc+aR)/aR\subseteq\Soc(R/aR).
\]
It remains to prove that $\Soc(R/aR) \subseteq (\m^s+aR)/aR$.

Choose $b\in R$ such that $b+aR$ is a nonzero socle element. Thus
$b\notin aR$ and $\m b\subseteq aR$. Since $\m^{s+1}\subseteq aR$,
there is an integer $0\le n\le s$ such that
\[
 b\in\m^n+aR,\qquad b\notin\m^{n+1}+aR.
\]
If $n=0$, then $b$ is a unit and $\m b\subseteq aR$ gives $\m=aR$,
contrary to the fact that $R$ is not Gorenstein. Subtracting an element of
$aR$, we may assume $b\in\m^n$ with $1\le n\le s$.
We show that $n<s$ is impossible.

Suppose $n<s$. By Lemma~\ref{lem:powers}(2), we may write
\[
 b=a^n\sum_{i=0}^{nq}r_i u^i\qquad(r_i\in R).
\]
For $0\le i\le(n-1)q$, we have $a^nu^i\in a(aK^q)^{n-1} = a\m^{n-1}\subseteq aR$.
If all coefficients $r_i$ with $(n-1)q<i\le nq$ belonged to $\m$,
we would have $b\in aR+\m^{n+1}$, a contradiction.
Choose the largest $d\in\{(n-1)q+1,\ldots,nq\}$ with $r_d\notin\m$.
Then $1\le nq+1-d\le q$, so $au^{nq+1-d}\in aL^q=\m$.
Since $b\m\subseteq aR$ and $b\in\m^n$,
Lemma~\ref{lem:powers}(2) gives
\[
 (au^{nq+1-d})b\in aR\cap\m^{n+1}=a\m^n.
\]
Therefore there are $r'_0,\ldots,r'_{nq}\in R$ such that
\[
 a^{n+1}\sum_{i=0}^{nq}r_i u^{i+nq+1-d}
 =a^{n+1}\sum_{i=0}^{nq}r'_i u^i.
\]
Cancel the regular element $a^{n+1}$ and reduce modulo $aS$.
Since $r_i\in\m$ for $d<i\le nq$ and $\m S=aS$, we obtain
\[
 \sum_{i=0}^{d}\overline{r_i}\,\overline{u}^{i+nq+1-d}
 =\sum_{i=0}^{nq}\overline{r'_i}\,\overline{u}^i.
\]
Since $n<s$, we have $nq+1<e$. Independence of
$1,\overline{u},\ldots,\overline{u}^{nq+1}$
now gives $\overline{r_d}=0$, a contradiction. Thus $n=s$, proving (1).

(2) For the dimension calculation, multiplication by $a^s$ induces a
surjective $k$-linear map
\[
 S/aS\longrightarrow(\m^s+aR)/aR,\qquad
 x+aS\longmapsto a^sx+aR.
\]
This is well defined because $a^{s+1}S\subseteq aR$; its target is
annihilated by $\m$ by (1). Since
\[
 aR\cap\m^s=a\m^{s-1}=a^sL^{(s-1)q},
\]
the kernel is the image of $L^{(s-1)q}$ in $S/aS$, namely
$\langle1,\overline{u},\ldots,\overline{u}^{(s-1)q}\rangle_k$.
Consequently,
\[
 (\m^s+aR)/aR\cong
 \langle1,\overline{u},\ldots,\overline{u}^{e-1}\rangle_k
 /\langle1,\overline{u},\ldots,\overline{u}^{(s-1)q}\rangle_k.
\]
The numerator has dimension $e$, and the denominator has dimension
$(s-1)q+1$, since $(s-1)q<e-1$. This proves (2).
\end{proof}

We now complete the proof of Theorem~\ref{thm:main}.

\begin{proof}[Completion of the proof of Theorem~\ref{thm:main}]
Assume (1)--(3).
By Lemma~\ref{lem:socle}(1) and the formula for $\Soc(R/aR)$ above,
\[
 R:\m=R+a^{-1}\fkc.
\]
Therefore
\[
 K\cap aS=K:(R+a^{-1}\fkc)=K:(R:\m)=\m K.
\]
Since $\m S=aS$, this proves the intersection
assertion in (b).

The image of $K$ in $S/aS$ has basis $1,\overline{u}$, so
$K/\m K\cong(K+aS)/aS$ has dimension two. Hence $\type(R)=2$,
and Nakayama's lemma gives $K=R+Ru$. Together with $e>v$ and the
reduction property of $(a)$, Lemma~\ref{lem:powers}(1) gives (4).

Since $a$ is a parameter, $\dim_k\Soc(R/aR)=\type(R)=2$.
Lemma~\ref{lem:socle}(2) therefore gives
\[
 e=(s-1)q+3=(s-1)(v-1)+3.
\]
Together with $e>v$, this proves (a).
Finally $K^{e-1}=(R+Ru)^{e-1}=S$ by Lemma~\ref{lem:powers}(1).
Fix an integer $d$ with $0\le d<e-1$. The image of
$K^d$ in $S/aS$ is spanned by $1,\overline{u},\ldots,\overline{u}^d$ and has
dimension $d+1<e$. Thus $K^d\ne S$, proving the rest of (b).
\end{proof}

\section{Numerical semigroup rings}

We now apply Theorem~\ref{thm:main} under the additional assumption
that the ring $R$ in Notation~\ref{not:base} has the form $R=k[[H]]$
for a numerical semigroup $H$.
Write 
$$H=\langle a_1,\ldots,a_v\rangle = \{\lambda_1a_1 + \lambda_2 a_2  + \cdots + \lambda_v a_v\mid 0 \le \lambda_1, \lambda_2, \ldots , \lambda_v \in \mathbb{Z}\}$$
 with minimal generators
$0<a_1<\cdots<a_v$. Then $\m=(t^{a_1},\ldots,t^{a_v})$,
$v=\mu_R(\m)$, and $e=a_1$.
Let $F = \max(\mathbb{Z} \setminus H)$ denote the Frobenius
number of $H$. For $0\ne f\in k[[t]]$, write $\ord_t(f)$ for
the least exponent of $t$ with a nonzero coefficient in $f$. Put
\[
 \PF(H)=\{f\in\mathbb Z\setminus H\mid f+h\in H
                  \text{ for every }h\in H\setminus\{0\}\}.
\]
By the description of the canonical module in
\cite[Example~(2.1.9)]{GW78}, after multiplying by $t^F$, we may take
\[
 K=\sum_{f\in\PF(H)}Rt^{F-f}.
\]
In particular, $\type(R)=|\PF(H)|$.

\Needspace{17\baselineskip}
We recall the consequence of Matthews' formula used below.

\begin{lemma}[{\cite[Lemma~2.7]{Matthews04}}]\label{lem:matthews}
Suppose $H=\langle e,e+d,\ldots,e+(v-1)d\rangle$ is minimally generated,
where $d>0$, $v\ge3$, and $e>v$. Put
\[
 b=\left\lfloor\frac{e-2}{v-1}\right\rfloor,
 \qquad r=e-1-b(v-1).
\]
Then the following assertions hold true.
\begin{enumerate}[label=\textup{(\arabic*)}]
\item $\PF(H)=\{be+jd\mid b(v-1)<j\le e-1\}$ and
      $K=R+Rt^d+\cdots+Rt^{(r-1)d}$.
\item $\type(R)=r$, the unique integer in $[1,v-1]$ congruent to
      $e-1$ modulo $v-1$. In particular,
      $\type(R)=2$ if and only if $e\equiv3\pmod{v-1}$.
\end{enumerate}
\end{lemma}

From the above lemma, we readily obtain the following as a corollary of Theorem~\ref{thm:main}.

\begin{corollary}\label{cor:semigroup}
The following conditions are equivalent.
\begin{enumerate}[label=\textup{(\arabic*)}]
\item $\tr(K^{v-1})=\m$ and $\tr(K^v)\subsetneq\m$.
\item $K^{v-1}\cong\m$ and $K^v\not\cong\m$ as $R$-modules.
\item $H=\langle e,e+d,\ldots,e+(v-1)d\rangle$ for some $d>0$,
      with $e>v$ and $\type(R)=2$.
\item $H=\langle e,e+d,\ldots,e+(v-1)d\rangle$ for some $d>0$,
      with $v\ge3$, $e>v$, and $e\equiv3\pmod{v-1}$.
\end{enumerate}
\end{corollary}
\begin{proof}
Conditions (1) and (2) are equivalent by Theorem~\ref{thm:main}.

Assume (1). Theorem~\ref{thm:main} gives $\type(R)=2$ and $e>v$.
Choose $d>0$ such that $K=R+Rt^d$. By Theorem~\ref{thm:main}(4),
choose $a\in\m$ such that $\m=aK^{v-1}$ and $(a)$ is a reduction of $\m$.
Then $\ord_t(a)=e$, since $1\in K^{v-1}\subseteq k[[t]]$. As
$K^{v-1}=\sum_{i=0}^{v-1}Rt^{id}$, we have
\[
 H\setminus\{0\}=\bigcup_{i=0}^{v-1}(e+id+H).
\]
Thus $H=\langle e,e+d,\ldots,e+(v-1)d\rangle$.
Since $H$ has $v$ minimal generators, this generating set is minimal.
This proves (3).

Assume (3), and choose $d>0$ as there. Lemma~\ref{lem:matthews}
gives $K=R+Rt^d$, and hence
\[
 \m=t^e(R+Rt^d+\cdots+Rt^{(v-1)d}),\qquad S=R[t^d].
\]
Obviously, $(t^e)$ is a reduction of $\m$ because the order $e$ is the multiplicity of $R$.
Hence Theorem~\ref{thm:main}(4) implies (1).
Finally, (3) and (4) are equivalent by Lemma~\ref{lem:matthews}(2).
\end{proof}

For comparison with the last possible exponent $v-1$, the traces can
be computed at every exponent for the arithmetic progressions above.
The case $j=1$ of the following proposition follows from the nearly
Gorenstein property proved by Herzog, Hibi, and Stamate
\cite[Proposition~1.4]{HHS21}.

\begin{proposition}\label{prop:arithmetic-traces}
In the setting of Lemma~\ref{lem:matthews}, suppose $r\ge2$.
For every integer $j\ge1$,
\[
 \tr(K^j)=\m\quad\Longleftrightarrow\quad j(r-1)\le v-1.
\]
\end{proposition}
\begin{proof}
Fix an integer $j\ge1$. By Lemma~\ref{lem:matthews},
$K^j=\sum_{i=0}^{j(r-1)}Rt^{id}$. Suppose $j(r-1)\le v-1$.
For $0\le i\le v-1-j(r-1)$, we have
$t^{e+id}K^j\subseteq\m$, so $t^{e+id}\in R:K^j$.
Multiplying these elements by $1,t^d,\ldots,t^{j(r-1)d}$ gives every
generator $t^e,t^{e+d},\ldots,t^{e+(v-1)d}$ of $\m$. Hence
\[
 \m\subseteq(R:K^j)K^j=\tr(K^j)\subseteq\m,
\]
where the last inclusion follows from Lemma~\ref{lem:trace}(1).

Suppose $j(r-1)\ge v$. Since $K^j\not\subseteq R$, we have
$R:K^j\subseteq\m$. Also $K^j\subseteq k[[t]]$. Thus, to prove
$t^e\notin\tr(K^j)$, it suffices to show that $R:K^j$ contains no
element of order $e$. If $x\in R:K^j$ has order $e$, then
$t^{vd}\in K^j$ gives $\ord_t(xt^{vd})=e+vd\in H$.
But then there are integers $p\ge1$ and $0\le h\le p(v-1)$ with
\[
 e+vd=pe+hd.
\]
If $p=1$, then $h=v$, a contradiction. Suppose that $p\ge2$. Then
$(p-1)e=(v-h)d>0$. Since $\gcd(e,d)=1$, we obtain
$e\le v-h\le v$, again a contradiction. Hence $R:K^j$ contains no
element of order $e$, and $t^e\notin(R:K^j)K^j=\tr(K^j)$.
\end{proof}

\begin{example}\label{ex:sharp}
For every $v\ge3$, let
\[
 H=\langle v+2,v+3,\ldots,2v+1\rangle
\]
and $R=k[[H]]$. The semigroup $H$ is minimally generated by the indicated
$v$ integers, since all are
smaller than twice the least generator. Its multiplicity is $e=v+2$,
so $e>v$ and $e\equiv3\pmod{v-1}$. By Corollary~\ref{cor:semigroup},
$\tr(K^{v-1})=\m$ but $R$ is not almost Gorenstein.
Thus the exponent $v$ in Proposition~\ref{prop:ag}(3) cannot be replaced
by $v-1$. For instance, when $v=4$, the ring of
$H=\langle6,7,8,9\rangle$ has
$K=R+Rt$ and its first three canonical powers have trace $\m$,
whereas the fourth does not.
\end{example}

\begin{example}
The condition $K^v\not\cong\m$ in Theorem~\ref{thm:main}(2) cannot be
omitted. Let $H=\langle4,6,7,9\rangle$ and $R=k[[H]]$. Then
$\PF(H)=\{2,3,5\}$, $K=R+Rt^2+Rt^3$, and
$\m K=\m$. Hence $R$ is almost Gorenstein. Moreover,
$K^2=S=t^{-4}\m$, so $K^3=K^4\cong\m$.
Since $v=4$, we have $K^{v-1}\cong\m$, whereas $\sigma(R)=\infty$.
\end{example}

\begin{example}
Let $H=\langle8,9,10,21\rangle$ and $R=k[[H]]$. Then $v=4$,
$\PF(H)=\{11,22,23\}$, and
\[
 K=R+Rt+Rt^{12},\qquad
 K^2=R+Rt+Rt^2+Rt^{13},\qquad
 \m=t^8K^2.
\]
Since $t^8\in R:K^2$, Lemma~\ref{lem:trace}(1) gives
\[
 \m=t^8K^2\subseteq(R:K^2)K^2=\tr(K^2)\subseteq\m.
\]
On the other hand, $t^3\in K^3$ and $11\notin H$,
so no element of $R:K^3$ has order eight. Since $R:K^3\subseteq\m$
and $K^3\subseteq k[[t]]$, we have $t^8\notin\tr(K^3)$.
Therefore $\sigma(R)=2=v-2$.
\end{example}

The index $\sigma(R)$ suggests a classification problem for nearly
Gorenstein rings beyond the case settled by Theorem~\ref{thm:main}.

\begin{conjecture}\label{conj:v-2}
Let $R=k[[H]]$ with $v\ge4$ and $\sigma(R)=v-2$. Then $\type(R)\in\{2,3\}$,
and $\type(R)=3$ if and only if $K^{v-2}\cong\m$, in which case
$\m=t^{e}K^{v-2}$.
\end{conjecture}

When $v=4$, the assertion $\type(R)\in\{2,3\}$ in
Conjecture~\ref{conj:v-2} follows from
\cite[Theorem~2.4]{MS21}, where Moscariello and Strazzanti proved that
nearly Gorenstein numerical semigroups with four minimal generators
have type at most three.

A separate question concerns the ring-theoretic meaning of arithmetic
progression generation. Corollary~\ref{cor:semigroup} answers this for
type two and multiplicity greater than $v$. If
$H=\langle e,e+d,\ldots,e+(v-1)d\rangle$ has type three and $e>v$,
Lemma~\ref{lem:matthews} and Proposition~\ref{prop:arithmetic-traces}
give
\[
 K=R+Rt^d+Rt^{2d},\qquad
 \tr(K^j)=\m\quad\Longleftrightarrow\quad 2j\le v-1
 \qquad(j\ge1).
\]

\begin{question}\label{q:arithmetic}
Which conditions on a one-dimensional Cohen--Macaulay local ring,
expressed in terms of its canonical fractional ideal and maximal ideal,
characterize arithmetic progression generation when specialized to
numerical semigroup rings of type three and multiplicity greater than
$v$?
\end{question}

\bibliographystyle{amsplain}
\bibliography{refs}
\end{document}